\documentclass [reqno] {amsart}

\usepackage{amsfonts}
\usepackage{amssymb}

\newtheorem{theorem}{Theorem}[section]
\newtheorem{corollary}[theorem]{Corollary}
\newtheorem{lemma}[theorem]{Lemma}

\theoremstyle{definition}
\newtheorem{example}[theorem]{Example} 

\theoremstyle{remark}
\newtheorem{remark}[theorem]{Remark} 

\numberwithin{equation}{section}

\begin{document}

\title[On the completeness of the root functions]
      {On the completeness of the root functions of the Sturm-Liouville problems for the Lam\'e
system in weighted spaces\footnote{This is a preprint version of the paper published in 
Z. Angew. Math. Mech., V. 95, N. 11 (2015), 1202--1214.  DOI 10.1002/zamm.201300303.}}


\author{A. Shlapunov}
\address[Alexander Shlapunov]
        {Siberian Federal University
\\
         Institute of Mathematics
\\
         pr. Svobodnyi 79
\\
         660041 Krasnoyarsk
\\
         Russia}
\email{ashlapunov@sfu-kras.ru}


\author{A. Peicheva}
\address[Anastasiya Peicheva]
        {Siberian Federal University
\\
         Institute of Mathematics
\\
         pr. Svobodnyi 79
\\
         660041 Krasnoyarsk
\\
         Russia}
				
\email{peichevaas@mail.ru}




\begin{abstract}
We consider three Sturm--Liouville boundary value problems (the coercive ones and
the non-coercive one) in a  bounded Lipschitz domain for the perturbed Lam\'e operator with the
boundary conditions of Robin type.  We prove that the problems are Fredholm ones in proper
weighted Sobolev type spaces. The conditions, providing the completeness of the root functions
related to the boundary value problem, are described.
\end{abstract}

\maketitle


\tableofcontents

\section{Introduction}

Investigating a boundary value problem, it is important to know both solvability
conditions and formulas for its exact and approximate solutions. For the linear problems,
the latter ones can be obtain with the use of expansions over the (generalized) eigenfunctions
related to the them (see, for instance, \cite{GokhKrei69}). Then, to use numerical methods
in the non-selfadjoint case, one needs to prove the completeness of the system of the
corresponding root elements. The results
of this kind are well known for the coercive (elliptic) boundary problems over smooth domains
(see \cite{Agmo62}, \cite{Brow59b}, \cite{Keld51}). For the Spectral Theory
related to the elliptic problems in Lipschitz domains we refer to the survey \cite{Agra11a}.
The root elements of general elliptic problems in weighted Sobolev spaces over domains
with the conic and edge singularities were studied in \cite{EgorKondSchu01}, \cite{Kond99},  
\cite{Tark06}.

Non-coercive boundary value problems for elliptic differential operators were discovered in the
middle of XX-th century (see, \cite{ADN59}, \cite{KN65}). 
In the Elasticity Theory, the problems of this kind  were indicated in
\cite{Ca59}, \cite{Ca60}. Considering the non-coercive problems, we essentially enlarge the class
of boundary conditions for which the completeness of the root elements holds true. This may lead
to a loss of the regularity of solutions to the problem near the boundary, but this is motivated
by the very nature of the problems (cf. \cite{PolkShla13}, \cite{TarShla13a}). 

The aim of this paper is the proof of the completeness in weighted Sobolev type spaces of the
root elements of three Sturm-Liouville problems  for the perturbed Lam\'e operator with the
boundary conditions of Robin type. The use of the  weighted spaces allows us to choose the
solutions with prescribed asymptotic behavior near the singular points of the boundary.

\section{Function spaces}
\label{required}
\label{s.3}

Let $D$ be a bounded domain in the Euclidean space $\mathbb{R}^{m}$, $m\geq 2$, with a
Lipschitz boundary. We consider complex-valued functions defined over the domain $D$ and its
closure $\overline{D}$. For $s \in {\mathbb Z}_+$  and $M \subset \overline D$, denote
by $C^s (\overline D, M)$ the set of $s$ times continuously differentiable functions over
$\overline D$ vanishing in a neighborhood of $\overline M$.  Let $C^\infty_{0} (D)$ be the space
of smooth functions with compact supports in $D$. The H\"older class with the exponent $0<\alpha
\leq 1$ over the set $M \subset {\mathbb R}^m$ we denote by $C^{0,\alpha} (M)$.
We write $L^q (D)$ for the standard Lebesgue space ($1\leq q \leq + \infty$).
We also write $H^s (D)$, $s \in \mathbb N$, for the Sobolev space of functions with all the weak
derivatives up to order $s$ belonging to $L^2 (D)$. Let $H^s_0 (D)$ stand for the closure of
$C^\infty_{0} (D)$ in $H^s (D)$. For positive non-integer $s$ we denote by $H^s (D)$ the
Sobolev-Slobodetskii space, see, for instance, \cite{LiMa72}.

Considering the spaces with negative smoothness we use the following standard construction. Let
$H^+$ and $H^0$ be complex Hilbert spaces with inner products $(\cdot, \cdot)_+$ and $(\cdot,
\cdot)_0$ respectively. Assume that $H^+$ is embedded continuously into $H^0$ and denote by $J_0:
H^+ \to H^0$ the corresponding  embedding. Moreover, we assume that $H^+$ is dense in $H^0$.
Then  let $H  ^{-} $ stand for the completion of $H^+$ with respect to the norm
$ \| u \|_{-} = \sup_{\substack{v \in H^+ \\ v \ne 0}}
\frac{|(v,u)_{0}|}{\| v \|_{+}}$.

The following lemma is well known (see, for instance, \cite[\S 3]{Sche60}).

\begin{lemma}
\label{l.dual}
The Banach space $H^{-} $ is topologically isomorphic to the dual space $(H^{+}) '$. Besides, the
isomorphism is defined by the Hermitian form $\langle v,u \rangle = \lim_{\nu \to \infty}
(v,u_\nu)_{0}$, $u \in H^{-}$, $v \in H^{+}$, where  $\{ u_\nu \}$ is a sequence in $H^+$
converging to $u$ in $H^{-}$. Moreover, if the embedding $J_0: H^+\to H^-$ is compact then the
space $H^{0}$ is compactly embedded to $H^{-}$.
\end{lemma}

Thus, $H^{-s} (D)$, $s>0$, corresponds to $H^{-}$ if $H^{0}=L^2 (D)$, $H^+=H^s(D)$. If
$H^{0}=L^2 (D)$, $H^+=H^s_0(D)$ the space $H^{-}$ will be denoted by $\tilde H^{-s} (D)$.

The weighted spaces appears naturally during the investigation of mixed boundary problems because
the weight can be used to control the behavior of the solutions near the set where the boundary
conditions change the character. Choose a closed set $Y $ on $\partial D$. In
order to control the growth of functions near $Y$ we introduce the weighted spaces associated to
$Y$. Assume that  $\rho \in C (\overline D)$ is a $C^1$-smooth function over $\overline D
\setminus Y$ such that  $0\leq \rho (x)\leq 1$, $x \in \overline D$, $\frac{\partial
\rho}{\partial x_j} \in L^\infty (D)$, $1\leq j \leq m$ and  $\rho (x)=0$ if and only if $x \in
Y$. In particular, $\rho\equiv 1$ will correspond to the usual Sobolev spaces. If $Y\ne
\emptyset$, then in typical situations,  for domains with piece-wise smooth boundaries, the
function $\rho (x)$ is the distance form the point $x \in \overline D$ to the singular set $Y
\subset \partial D$.

Now, for $\gamma \in {\mathbb R}$, $s=0$ and $s=1$ we introduce the weighted Sobolev spaces
$H^{s,\gamma} (D)$  as the completion of $C^s (\overline D, Y)$ with respect to the norms,
induced by the following scalar products:
\[
( u,v)_{H^{s,\gamma}   (D)} = \sum_{|\alpha|\leq s}    \Big(\rho^{|\alpha|-\gamma-s}\partial
^\alpha u ,\rho^{|\alpha|-\gamma-s}\partial ^\alpha v\Big)_{L^{2} (D)} , s=0,1,
\]
(cf. \cite[\S 1.7]{BoKo} for the localized situation where the
weight is given in local coordinates near the singularity). Moreover, for $0<s<1$ we introduce
the weighted Sobolev-Slobodetskii spaces as the completion of  $C^{1} (\overline D, Y)$  with
respect to the norms, induced by the following scalar product:
\[
( u,v)_{H^{s,\gamma}   (D)} =  ( u,v)_{H^{0,\gamma+s}   (D)}+ (\rho^{-\gamma} u,
\rho^{-\gamma}v)_{H^{s} (D)}.
\]
Similar fractional  weighted spaces were considered in \cite{Kond66} for the
localized situation.

As before, the weighted negative Sobolev-Slobodetskii space $H^{-s,\gamma} (D)$, $0<s\leq 1$,
will be defined as the space $H^{-}$ for $H^{0}=H^{0,\gamma} (D)$, $H^+=H^{s,\gamma} (D)$.

\begin{lemma} \label{eq.l.emb.s.gamma}
For each fixed $\gamma \in {\mathbb R}$ the space $H^{s,\gamma} (D)$ is compactly embedded into
$H^{s',\gamma} (D)$, if $-1 \leq s'<s\leq 1$. Moreover, if $1/2 <s \leq 1$ the trace operator
$tr: H^{s,\gamma} (D) \to H^{s-1/2,\gamma} (\partial D)$ is correctly defined and bounded.
\end{lemma}

\begin{proof} The proof is standard. It is based on the Rellich-Kondrashov Theorem and the Trace
Theorem for the usual Sobolev spaces (see \cite[\S 1.7]{BoKo}). 
\end{proof}

Everywhere below, for a set $M \subset \overline D$ we denote by $H^{s,\gamma} (D,M)$
the completion of $C^{s} (\overline D,M \cup Y)$ in  $H^{s,\gamma} (D)$. In particular,
$H^{1} (D,\partial D)= H^{1} _0(D)$.

Besides, for a function space ${\mathfrak B}(D)$ over $D$ denote by $[{\mathfrak B }(D)]^k$ the
space of $k$-vector functions $u$ with the components  $u_j \in {\mathfrak B}(D)$. If ${\mathfrak
B}(D)$ is a normed space then we endow the space $[{\mathfrak B }(D)]^k$ with the norm
$\|u\|_{[{\mathfrak B }(D)]^k} = \Big(\sum_{j=1}^k \|u_j\|^2_{{\mathfrak B }(D)}\Big)^{1/2}$.
Thus, $[{\mathfrak B }(D)]^k$ is a Hilbert space if  the space ${\mathfrak B }(D)$ is a Hilbert
one.

\section{The Sturm Liouville-problem for the Lam\'e type system}
\label{s.4}

Fix an open connected set $S$ with piece-wise smooth boundary  $\partial S$ on the hypersurface
$\partial D$, a set $Y\subset \partial S$ and a weight $\rho$ associated with them. Denote by
${\mathcal L}$ the Lam\'e type operator in ${\mathbb R}^m$:
\begin{equation} \label{eq.Lame}
{\mathcal L}_0 (x,\partial)  = -  \mu (x) I_m \Delta_m  - (\lambda (x) +
\mu (x) ) \nabla_m \mbox{div}_m
\end{equation}
where $I_m$ is the identity $(m\times m)$-matrix, $\Delta_m$ is the Laplace operator in
${\mathbb R}^m$, $\nabla_m$ is the gradient operator in ${\mathbb R}^m$, $\mbox{div}_m $ is the
divergence operator in ${\mathbb R}^m$,
and  $\mu$, $\lambda$ are real-valued functions from $L^\infty (D)$ such that $\mu \geq \kappa $,
$(2\mu +\lambda) \geq \kappa$ for a constant $\kappa>0$. If $m=3$ and $\mu \geq 0$, $\lambda
\geq 0$  this operator plays an essential role in the description of the displacement of an
elastic body under the load (see \cite{Fi72}). It also can be used as a linearization of the
stationary version of the Navier-Stokes' type equations for viscous compressible fluid if the
pressure is known (see \cite[\S 15]{LaLi59}).

Clearly, the Lam\'e type operator is strongly elliptic and, if the functions $\mu$, $\lambda$
belong to $C^{0,1} (D)$ then  there is a formally non-negative self-adjoint
operator ${\mathcal L}_{\mathfrak D} (x,\partial) ={\mathfrak D}^*  {\mathfrak D} $ that differs
from ${\mathcal L}_0 (x,\partial)$ by the low order summands; here ${\mathfrak D} =
\sum_{j=1}^m {\mathfrak D}_j \partial _j   $ is a differential $(k\times m)$-matrix
first order operator and ${\mathfrak D}^*$ is its formal adjoint one. Of course, there are many
such operators ${\mathfrak D} $. To introduce three of them we denote by
 $M_1 \otimes M_2$ the Kronecker product of matrices $M_1$ and $M_2$,
 by $\mbox{rot}_m$  we denote  $\Big(\frac{(m^2-m)}{2}\times
m\Big)$-matrix operator with the lines   $\vec{e}_i \frac{\partial}{\partial x_j}-\vec{e}_j
\frac{\partial}{\partial x_i }$,  $1\leq i<j \leq m$,  representing the vorticity (or the
standard rotation operator for $m=2$, $m=3$), and by  ${\mathbb D}_m$ we denote
$\Big(\frac{(m^2+m)}{2}\times m\Big)$-matrix operator with the lines  $\sqrt{2} \vec{e}_i
\frac{\partial}{\partial x_i}$, $1\leq i \leq m$, and $\vec{e}_i \frac{\partial}{\partial
x_j}+\vec{e}_j \frac{\partial}{\partial x_i }$  with $1\leq i<j \leq m$,
representing the deformation (the strain). The we set:
\begin{equation} \label{eq.factor.1}
{\mathfrak D}^{(1)} =
\left(\begin{array}{lll} \sqrt{\mu } \ {\mathbb D}_m \\ 
\sqrt{\lambda } \mbox{div}_m  \\ \end{array} \right), \quad
{\mathfrak D}^{(2)}  = \left(\begin{array}{lll} \sqrt{\mu } \ \nabla_m \otimes I_m \\  \sqrt{\mu
+\lambda } \ \mbox{div}_m , \end{array} \right), \quad {\mathfrak D}^{(3)} =
\left(\begin{array}{lll} \sqrt{\mu } \ \mbox{rot}_m \\ \sqrt{2\mu +\lambda } \mbox{div}_m  \\
\end{array} \right),
\end{equation}
here $\lambda \geq 0$,  $k_1=(m^2+m)/2 + 1$ for the first operator,   $(\mu+\lambda)\geq 0$,
$k_2=m^2+1$  for the second operator, and $(2\mu+\lambda)\geq \kappa >0$,
$k_3=(m^2-m)/2 + 1$  for the third operator.

Thus, everywhere below we assume that  $\mu, \lambda \in C^{0,1 } (D) \cap L^{\infty} (D)$,
$\rho \nabla_m \mu , \rho \nabla_m \lambda \in [L^\infty (D)]^m$.

Consider a $(m\times m)$-matrix linear differential operator $A$ in the domain $D$ associated
with the operator ${\mathcal L}_{\mathfrak D}$, where $\mathfrak D$ is one
of the operators ${\mathfrak D}^{(j)}$, $j=1,2,3$:
\begin{equation} \label{eq.A.D}
A  u = {\mathfrak D}^* {\mathfrak D} u  +   a_1 \nabla_m \otimes I_m
+ a_0 (x) u,
\end{equation}
here $a_0$ and $a_1$ are functional $(m\times m)$- and $(m\times m^2)$- matrices
respectively with the components $a^{(p,q)}_j$ satisfying the following assumptions:
$\rho^2 a^{(p,q)}_0 \in L^{\infty} (D)$, $\rho a^{(p,q)}_1 \in L^{\infty} (D)$.

Let $\nu_ {\mathfrak D} =  \sum_{j = 1}^m {\mathfrak D}^*_j \nu_j
{\mathfrak D} $ be the conormal derivative with respect to the operator ${\mathfrak D}$, where
$\nu  = (\nu_1 , \dots \nu_{m} )$ is the field of the exterior unit normal vectors with respect
to $\partial D$ (defined for almost all points $x\in \partial D$).  Clearly, two operators
of the type $\nu_ {\mathfrak D}$ above, differ on a matrix with entries being
tangential derivatives with respect to the boundary.

Consider now the boundary operator
\begin{equation*} \label{eq.B}
B  = b_1 (x)\nu_ {\mathfrak D} + b_0 (x) + \partial _\tau
\end{equation*}
where $\partial_\tau $ is a  $(m\times m)$-matrix of tangential derivatives with
respect to $\partial D$.  As for the $(m\times m)$-matrices $b_0 (x)$
and $b_1 (x)$, we will assume  that their components are locally bounded functions on $\partial
D\setminus Y$.  We allow for the matrix $b_1 (x)$ to degenerate on $S$; in
this case we assume that $b_0 (x)$ is not degenerate on $S$ and the components of the tangential
part $\partial_\tau$ equal to zero on $S$.

\begin{remark} \label{r.stress.op} Usually, the first order boundary conditions related to
boundary problems for the Lam\'e operator are defined with the use of the stress boundary tensor
$\sigma$ with the components
\begin{equation} \label{eq.stress.tensor}
\sigma _{i,j}   =    \mu  \, \delta_ {i,j} \sum_{k=1}^m \nu_k \frac{\partial }{\partial x_k}  +
\mu \, \nu_j \frac{\partial }{\partial x_i}  + \lambda \, \nu_i \frac{\partial }{\partial x_j}  ,
\ 1\leq i,j \leq m.
\end{equation}
Then, with the tangential part $\partial_{\tau_0} = \left( (\nu(x) \rm{div}_m )^T -  \nu (x)
\rm{div}_m \right)$, we have
\begin{equation} \label{eq.stress.var}
\sigma   = \nu_{{\mathfrak D}^{(1)}} = \nu_{{\mathfrak D}^{(2)}} +  \mu (x) \partial_{\tau_0} =
\nu_{{\mathfrak D}^{(3)}}    + 2\mu (x) \partial_{\tau_0} .
\end{equation}
\end{remark}

We will study the following mixed problem:  given generalized $m$-vector function
$f$ in $D$, find a $m$-vector distribution  $u$ in $D$ satisfying in a proper sense 
(cf. \cite[\S 12]{Fi72} for ${\mathfrak D} = {\mathfrak D}^{(1)}$)
\begin{equation}
\label{eq.SL}
   \left\{
\begin{array}{rclcl}
     A u   & =  & f    & \mbox{ in }  & D, \\
     B u   & =   & 0  & \mbox{ on }  & \partial D.
  \end{array}   \right.
\end{equation}
If $S=\partial D$ then we obtain the classical Dirichlet problem for the strongly elliptic
operator ${\mathfrak D}^*{\mathfrak D}$. As it is well known, it is coercive due to the
G\aa{}rding inequality (see, for instance, \cite{Fi72}, \cite{LadyUral73}, \cite{LiMa72}).
That is why we will be concentrated on the case where $S\ne\partial D$.

The boundary problem (\ref{eq.SL}) related to ${\mathfrak D}= {\mathfrak D}^{(3)}$
was discovered by S.~Campanato (see \cite{Ca59}, \cite{Ca60}). However he proved
an Existence Theorem for it in the coercive case $S=\partial D$ only.

In the classical Theory of Boundary Value Problems, a typical assumptions
are the fulfillment of the Shapiro-Lopatinsky conditions for the pair
$(A,B)$ on the smooth part of $\partial D $ (see, for instance, \cite{Eski73}, 
\cite{LadyUral73}, \cite{LiMa72},  and others), that is a necessary
for the problem to be coercive. We will show below that for $S\ne\partial D$
and ${\mathfrak D}= {\mathfrak D}^{(1)}$ or ${\mathfrak D}= {\mathfrak D}^{(2)}$ the mixed
problem (\ref{eq.SL}) is coercive in the Sobolev spaces, but for  $S\ne\partial D$ and
${\mathfrak D} = {\mathfrak D}^{(3)}$ it is not (cf. \cite{Ca59} for $n=2$).

As we plan to use the perturbation method for compact self-adjoint operators,
we split the coefficients $a_0$ and $b_0$:
\[
a_0 =a_{0,0} + \delta a_0,  \, b_0  = b_{0,0} +\delta b_0,
\]
where $a_{0,0} (x)$ is a Hermitian non-negative functional $(m\times m)$-matrix over $D$ with
the components satisfying  $\rho^2 a^{(p,q)}_{0,0} \in L^{\infty} (D)$, and where $(m\times
m)$-matrix  $b_{0,0}$ is chosen in such a way that  $b_1 ^{-1} b_{0,0} $ is
Hermitian non-negative  functional matrix over $\partial D$.

Consider the following Hermitian forms on the space $[H^{1} (D,S \cup Y)]^m$:
\[
   (u,v)_{+,\gamma, {\mathfrak D}^{(j)}}  =  \left({\mathfrak D}^{(j)} u, {\mathfrak D}^{(j)}v
\right)_{[H^{0,\gamma} (D)]^k} +  (a_{0,0} u, v)_{[H^{0,\gamma} (D)]^m} +
 (b_1^{-1} b_{0,0}  u, v)_{[H^{0,\gamma} (\partial D \setminus S)]^m} .
\]
The form $(\cdot,\cdot)_{+,\gamma, {\mathfrak D}^{(2)}}$ is strongly coercive, i.e.
\begin{equation} \label{eq.strong.coercive}
\|{\mathfrak D}^{(2)} u \|^2 _{[H^{0,\gamma} (D)]^{k_1}} \geq c
\|\nabla_m u_j \|^2_{H^{0,\gamma} (D)} \mbox{ for all } u \in [H^1 (D)]^m
\end{equation}
with a constant $c$ being independent on $u$. The forms corresponding
to operators ${\mathfrak D}^{(1)}$ and  ${\mathfrak D}^{(3)}$ are not strongly
coercive because ${\mathfrak D}^{(1)} u =0$ with non-constant
vector $u=x_i \vec{e}_j - x_j \vec{e}_i$, $i\ne j$ and ${\mathfrak D}^{(3)} \nabla _m h =0$ in
$D$ for any harmonic function $h$ in $D$.

Denote by $H^{+,\gamma}_{{\mathfrak D}^{(j)}} (D)$, $j=1,2,3$, the completion of  $[H^{1} ( D,S
\cup Y)]^k$ with respect to the norm $\|\cdot\|_{+,\gamma,{\mathfrak D}^{(j)}}$ induced by the
inner product $(\cdot , \cdot )_{+,\gamma,{\mathfrak D}^{(j)}}$ (of course, if it is an inner
product).

\begin{lemma} \label{l.factor} The Hermitian form $(\cdot,\cdot)_{+,\gamma,{\mathfrak D}^{(j)}}$
defines an inner product on $[H^{1} (D,S \cup Y)]^m$ if one of the following
conditions holds true:

1) the open set $S \subset \partial D$ is not empty (in the topology of $\partial D$);

2) $ a_{0,0} \geq c_0 I_m \mbox{ in } \overline U$ with a constant $c_0>0$ on a non-empty
open set  $U \subset D$;

3) $ b_1^{-1} b_{0,0} \geq c_1 I_m  \mbox{ in } \overline V$ with a constant
$c_1>0$ on a non-empty open set  $V \subset \partial D\setminus S$.

Besides, in these cases we have:

a) the space  $[H^{1,\gamma} (D,S)]^m$ is continuously embedded in $H^{+,\gamma}_{\mathfrak D}
(D)$ if the components of the matrix $\rho b_1^{-1} b_{0,0}$ belong to
$L^\infty (\partial D  \setminus S)$;

b) the elements of $H^{+,\gamma}_{{\mathfrak D}^{(j)}} (D)$ belong to $[H^{1}_{\mathrm{loc}} (D
\cup S,S)]^m$.
\end{lemma}

\begin{proof} Regarding the statement on the scalar product, we only need to check that
$(u,u)_{+,\gamma,{\mathfrak D}}=0$ for $u \in [H^{1} (D,S \cup Y)]^m$ implies $u=0$ in $D$.

But the first order operators
\[
\left( \begin{array}{cc} \sqrt{\mu^{-1}} I_{k_1-1}& 0 \\ 0 & \sqrt{\lambda^{-1}}\\
 \end{array}\right) {\mathfrak D}^{(1)},
 \quad
\left( \begin{array}{cc} \sqrt{\mu^{-1}} I_{k_2-1} & 0 \\ 0 &
\sqrt{(\mu+\lambda)^{-1}} \\ \end{array}\right){\mathfrak D}^{(2)}, 
\]
\[
\left( \begin{array}{cc} \sqrt{\mu^{-1}} I_{k_3-1}& 0 \\ 0 & \sqrt{(2\mu+\lambda)^{-1}}\\
 \end{array}\right) {\mathfrak D}^{(3)},
\]
have constant coefficients and
injective principal symbols. Then Petrovskii Theorem implies that the distributions-solutions
to ${\mathfrak D}^{(j)} u =0$ in $D$, $j=1,2,3$, are real analytic there. Hence the
statement on the scalar product under condition 2) follows from the Uniqueness Theorem for real
analytic functions.

Then it follows from conditions  1) or 3) that any vector $ u\in[H^1 (D, S\cup Y)]^m$ satisfying
$(u,u)_{+,\gamma,{\mathfrak D}}=0$ vanishes on an open non-empty subset of $\partial D $.
As  $u$ also satisfies ${\mathfrak D}^{(j)} u =0$ in $D$,
the Uniqueness Theorem for the Cauchy problem for systems with injective symbols
implies that $u\equiv 0$ in $D$ (see, for instance, \cite[Theorem 2.8]{ShTaLMS}).

As the solutions of the system ${\mathfrak D}^{(j)}u=0$ in $D$ real analytic there, then there
are no such solutions of the class $[H^1 (D, \partial D)]^m$. Then it follows from  the
G\aa{}rding inequality for the Hermitian form $({\mathfrak D} \cdot, {\mathfrak D} \cdot )_{[L^2
(D)]^k}$ induced by the strongly elliptic operator ${\mathfrak D}^*{\mathfrak D}$
that
\begin{equation} \label{eq.Ga}
\|u\|^2_{[H^{1} (D)]^m} \leq const \, ({\mathfrak D} u, {\mathfrak D} u )_{[L^2
(D)]^k} \mbox{ for all } u \in [H^1 (D, \partial D)]^m,
\end{equation}
the constant $const$ being independent on $u$. In particular, this implies that
the statement b) holds true

Finally, the statement a) can be checked directly.
\end{proof}

\begin{lemma} \label{l.factor.1} Let  $\mu, \lambda \in L^\infty (D)$, $(\mu+\lambda)\geq 0$.
If $\rho\equiv 1$ then the embedding $H^{+,\gamma}_{{\mathfrak D}^{(j)}} (D) \to
[ H^{1} (D)]^m$, $j=1,2$ is bounded under one of the conditions 1), 2), 3) of Lemma
\ref{l.factor}. If there is $q>0$ such that
\begin{equation} \label{eq.a}
\rho^2 a_{0,0} \geq q I_m \mbox{ in } \overline D \setminus Y,
\end{equation}
then the space $H^{+,\gamma}_{{\mathfrak D}^{(j)}} (D)$ is continuously
embedded into $[ H^{1,\gamma} (D)]^m$, $j=1,2$.
\end{lemma}

\begin{proof}
The continuity of the embedding $H^{+,\gamma}_{{\mathfrak D}^{(1)}} (D) \to
[ H^{1,\gamma} (D)]^m$ follows from the second Korn inequality (see \cite[formula
(12.11)]{Fi72}). The continuity of the embedding $H^{+,\gamma}_{{\mathfrak D}^{(2)}} (D) \to
[ H^{1,\gamma} (D)]^m$ follows from the strong coercive estimate (\ref{eq.strong.coercive}). 
\end{proof}

For the operator ${\mathfrak D}^{(3)}$ Lemma \ref{l.factor.1} is not true.

\begin{example}  { \rm Take the cylinder
\[
D = \{(x_1,x_2) \in \Omega, 0<x_3 <1 \}
\]
with the base  $\Omega = \{x^2_1 + x_2^2 < 1\}$
in ${\mathbb R}^3$ as the domain. Let, for instance, $\rho
\equiv 1$, $a_{0,0} \equiv \rho^{-2}$, $b_1^{-1} b_{0,0}\in L^{\infty} (\partial D\setminus S)$,
and $S =\emptyset$. Then $H^{1,\gamma} (D) = H^1 (D)$.  Set $ h_m =\Re (x_1+ \iota x_2)^m $,
with $\iota$ being the imaginary unit and $\Re(a)$ being the real part of a complex number
$a$. The function $ h_m$ is harmonic in $D$ as the real part
of the holomorphic monomial $(x_1+ \iota x_2)^m$. It is easy to check that
the system $\{\nabla_3 h_m \}$ is orthogonal in $[H^{1} (\Omega)]^3$ and
$[H^{1} (D)]^3$ (the last one follows from the Fubini Theorem). Then Bessel's Inequality
implies that the sequence $\{ u_m = \nabla_3 h_m /\|\nabla_3 h_m\|_{[H^1 (D)]^3}\}$
converges weakly to zero in $[H^{1} (D)]^3$. It follows from the Sobolev Embedding Theorem that
the sequence $\{ u_m \}$ converges to zero $[L^2 (D)]^3$ and $[L^2 (\partial D)]^3$
while $\|u_m\|_{[H^1 (D)]^3}=1$. Finally, as ${\mathfrak D}^{(3)} u_m=0$ we see that
$\|u_m\|_{H^{+,\gamma}_{{\mathfrak D}^{(3)}} (D)} \to 0$ if $m \to +\infty$. This means
that the continuous embedding $H^{+,\gamma}_{{\mathfrak D}^{(3)}} (D) \hookrightarrow
[H^{1,\gamma} (D)]^3$ is impossible. }
\end{example}

\begin{example} {\rm In order to illustrate the case $S \ne \emptyset$ we set $m=2$. In this
situation one can easily modify the famous Hadamard's example related to the ill-posed Cauchy
for the Laplace operator. Namely, take the upper half-circle  $\{x_2>0, x^2_1 + x_2^2 < 1\}$ as
the domain $D$, and take the interval $[-1,1]\subset Ox_1$ as  the set $S$. For instance,
let $a_{0,0} \equiv 0$, $b_{0,0} \equiv 0 $ on $\partial D\setminus S)$, $\rho \equiv 1$.
Note that the matrix $\left(\begin{array}{ll} \ \mbox{rot}_2 \\  \ \rm{div} _2
\end{array} \right)$ is adjoint to the Cauchy-Riemann system. On $\partial D$ we consider the
sequence  $\{ v_p \}$ with the components
\[
v^{(1)}_p (x)=\left\{ \begin{array}{ll} \frac{1}{p}\sin{(\pi p x_1)}, & x_2=0, \\
0, & x_2  >0. \end{array}\right. , \quad v^{(2)}_p \equiv 0.
\]
Obviously, each $v_p$ is a Lipschitz function on $\partial D$ and
the sequence $\{ v_p\}$ converges to zero in $[H^{1/2} (\partial D)]^2$.
If $P _\Delta $ stands for the Poisson  integral for the Dirichlet Problem for the Laplace
operator in  $D$ then the sequence $\{ P_\Delta(v_m)\}$ converges to zero
in  $[H^{1} (D)]^2$. Now it is clear that the functions
\[
\left\{ u_p =  \left( \begin{array}{ll} \Re(\sin{\pi (x_1 - \iota x_2))} \\
\Im(\sin{\pi (x_1 - \iota x_2))} \end{array} \right) - P _\Delta (v_p) \right\}
\]
belong to $[H^1 (D)]^2$ and they equal to zero on $S$ (here $\Im(a)$ denotes the imaginary
part of a complex number $a$). Moreover, by the construction,
the sequence $\{ {\mathfrak D}^{(3)}  u_p = - {\mathfrak D}^{(3)}  P_\Delta (v_p) \}$
converges to zero in $L^2 (D)$. That is why $\{ u_p\}$ converges to zero in
$H^{+,\gamma} _{{\mathfrak D}^{(3)}}(D)$ but it can not be convergent even
in $[L^2 (D)]^2$.}
\end{example}

However, one can indicate conditions providing useful embedding theorems for the spaces generated
by  non-coercive forms (see \cite{ADN59}). The following statement describes reasonable
assumptions for $H^{+,\gamma} _{{\mathfrak D}^{(3)}}(D)$ to be embedded into
Sobolev-Slobodetskii spaces. The scheme of its  proof is similar to the cases of scalar
operators (see \cite{PolkShla13}, \cite{TarShla13a}).

\begin{theorem}
\label{t.emb.half}
Let $\mu$, $\lambda$ belong to the class $C^\infty (X)$ with a neighborhood
$X$ of the compact $\overline D$, and let $\rho \equiv 1$. Then

1)  the space $H^{+,\gamma}_{{\mathfrak D}^{(3)}} (D)$ is  continuously embedded into
$[L^{2} (D)]^m$, if condition (\ref{eq.a}) holds true.

2) the space $H^{+,\gamma}_{{\mathfrak D}^{(3)}} (D)$ is continuously embedded into
$[H^{1/2-\varepsilon} (D)]^m$ with any $\varepsilon >0$ if
\begin{equation} \label{eq.b}
b_1^{-1} b_{0,0}  \geq c_1 I_m \mbox{ on } \partial D \setminus S \mbox{ with a constant
} c_1>0.
\end{equation}
Moreover, if  $\partial D \in C^2$, then (\ref{eq.b}) implies that  the
space $H^{+,\gamma}_{{\mathfrak D}^{(3)}} (D)$ is continuously embedded into
$[H^{1/2} (D)]^m$.
\end{theorem}
\begin{proof} The statement 1) is obviously true.

Let estimate (\ref{eq.b}) is fulfilled. Then the norm  $\|\cdot\|_{+,\gamma,
{\mathfrak D}^{(3)}}$ is not weaker than the norm $\|\cdot\|_h$ on $ [H^1 (D, S)]^m$, where
\begin{equation} \label{eq.h.norm}
\|u\|_h =
\Big(
    \| {\mathfrak D}^{(3)} u \|^2_{[L^{2} (D)]^k}
      + \|  u \|_{[L^2 (\partial D \setminus S)]^m}^2
   \Big)^{1/2}, \, u \in [H^1 (D, S)]^m.
\end{equation}
Fix a number $\varepsilon >0$.  Let us show that the norm  $\|\cdot\|_h$ is not weaker
than the norm $\|\cdot\|_{[H^{1/2-\varepsilon} (D)]^m}$ on $ [H^1 (D, S)]^k$. Indeed,
integrating  by parts it is easy to see that a vector function $v \in [C^\infty (X)_{0}]^m$
satisfy $({\mathfrak D}^{(3)})^* {\mathfrak D}^{(3)} v =0$ in $X$ if and only if $ {\mathfrak
D}^{(3)}v =0$ in $X$. As we have seen in the proof of Lemma \ref{l.factor}, any weak solution to
this equation  is real analytic in $X$ and hence we have $v\equiv 0$. Thus, under hypothesis
of the theorem the operator ${\mathcal L}_{{\mathfrak D}^{(3)}}$ has a two-sided fundamental
solution on $X$, say, $\phi_m (x,y)$. For instance, if $\mu$ and $\lambda$ are constants, we
may take the famous Kelvin-Somigliana kernel.

The volume potential
\begin{equation} \label{eq.VP}
\Phi v (x)= \int_{D} \phi_{m} (x,y) v (y) dx   , v \in [L^2 (D)]^m,
\end{equation}
induces the bounded linear operator $\Phi: [L^2 (D)]^m \to [H^2 (X)]^m$ for any bounded domain
$X$ containing $\overline D$.

It is clear that any element $u\in H ^{-s}(D)$ extends up to an element
$U \in H ^{-s}({\mathbb R}^{m})$ via
\[
\langle U, v\rangle_{{\mathbb R}^{m}} =\langle u, v\rangle _{D}
\mbox{ for all } v \in H^{s} ({\mathbb R}^{m});
\]
here $\langle \cdot, \cdot \rangle_D$ is the pairing on $H \times H'$ for a space
$H$ of distributions over $D$. It is natural to denote it by $\chi_{D} u$.
The  defined in this way linear operator
$\chi_{D}: H ^{-s}(D) \to H ^{-s}({\mathbb R}^{m})$, $s  \in
{\mathbb R}_+$ is obviously bounded. Since the distribution  $\chi_D u $ is supported in
$\overline D$, the volume potential
(\ref{eq.VP}) induces the bounded linear operator
\[
\Phi \circ \chi_{D}  I_m: [H^{\varepsilon-1/2} (D)]^m \to [H^{\varepsilon+3/2} (X)]^m, \quad
0<\varepsilon \leq 1/2,
\]
for any bounded domain $X$ containing $\overline D$ (see, \cite{Agra90}).

Hence, the operators
\[
{\mathfrak D} \circ \Phi \circ \chi_{D} I_m:
[H^{\varepsilon-1/2} (D)]^m \to [H^{\varepsilon+1/2}
(X)]^k,
\]
\[
 \nu_ {\mathfrak D} \circ \Phi  \circ \chi_{D} I_m:
 [H^{\varepsilon-1/2} (D)]^m \to [H^{\varepsilon} (\partial D)]^m
\]
are bounded, too, if  $0<\varepsilon \leq 1/2$ because of the Trace Theorem for the Sobolev
spaces. Note that for $\varepsilon =0 $ this statement is not true because the elements of the
space $H^{1/2} (X)$ may have no traces on $\partial D \subset X$.

Now integrating by parts  we obtain for
$u \in [H^1 (D, S)]^m$ and $v \in [L^2 (D)]^m$:
\begin{equation} \label{eq.emb.half.1}
(v, u)_{[L^2(D
)]^m} = ({\mathcal L}_{\mathfrak D}  \Phi I_m v, u)_{[L^2(D)]^m} =
\left( {\mathfrak D} \Phi I_m v , {\mathfrak D}  u \right)_{[L^2 (D)]^k} +
(\nu_{\mathfrak D} \Phi I_m v, u)_{[L^2 (\partial D\setminus S)]^m}.
\end{equation}

Take  a sequence  $\{ v_\mu \} \subset [H^1 (D)]^m$, converging to $v$ in the space
$[H^{\varepsilon-1/2} (D)]^m$, $0 < \varepsilon < 1/2$.
As the space $H^{s} (D)$ is reflexive for each $s$, using (\ref{eq.emb.half.1})
and the continuity of the operators ${\mathfrak D} \circ \Phi  \circ \chi_{D} I_m$,
$\nu_{\mathfrak D} \circ \Phi  \circ \chi_{D} I_m$ above, we obtain for
$u \in [H^1 (D, S)]^m$:
\[
c \|u\|_{[H^{1/2-\varepsilon} (D)]^m} \leq
  \left\| {\mathfrak D} \circ \Phi  \circ \chi_{D} I_m \right\|
 \left\|   {\mathfrak D}  u \right\|_{[L^2 (D)]^k} +
 \left\|\nu _{\mathfrak D}
   \circ \Phi \circ \chi_{D} I_m \right\|
    \left\| u \right\|_{[L^2 (\partial D \setminus S)]^m}
\]
with a constant $c>0$ being independent on $u$.
Thus, there are constant $C_1>0$, $C_2>0$ such that
\[
 \|u\|_{[H^{1/2-\varepsilon} (D)]^m} \leq
   C_1 \ \|u\|_h \leq
   C_2 \ \|u\|_{+,\gamma} \mbox{ for all } u \in [H^1 (\overline D, S)]^m.
\]
This proves the continuous embedding
$H^{+,\gamma} (D) \hookrightarrow [H^{1/2 -\varepsilon} (D)]^m$ with any $\varepsilon>0$.

Due to the factorization, the operator ${\mathcal L}_{\mathfrak D} $ is strongly elliptic
formally-selfadjoint and the Dirichlet problem for it is Fredholm of index zero (see, for
instance, \cite{Agra11c}, \cite[Lemma 3.2]{SchShTa03}).  As we noted above, ${\mathcal L}
_{{\mathfrak D}^{(3)}} u =0$ in  $D$ for $u \in C^\infty _{0} (D)$ if and only if  $u \equiv 0$. 
Therefore the Dirichlet problem for it is uniquely solvable. Let now $G$ and $P$ stand for the 
Green function and the Poisson integral of the Dirichlet Problem for the
Dirichlet problem for the operator ${\mathcal L}_{\mathfrak D}$ in $D$. Then they
induce the bounded operators (see, for instance,
\cite{Agra11c}, \cite[Theorem 3.3]{SchShTa03})
\[
 G_1: [\tilde H^{-1} (D)]^m \to [H^1_0 (D)]^m, \quad P_1: [H^{1/2} (D)]^m \to [H^1 (D)]^m.
\]
As the operator ${\mathcal L}_{\mathfrak D}$ extends to the continuous linear operator
${\mathcal L}_{\mathfrak D}:  [H^1 (D)]^m \to [\tilde H^{-1} (D)]^m$ via
\[
\langle {\mathcal L}_{\mathfrak D} u , v \rangle = ({\mathfrak D} u, {\mathfrak D} v )_{[L^2
(D)]^m}, u \in [H^1 (D)]^m, v \in [H^1_0 (D)]^m ,
\]
then $u=P_1u + G_1 {\mathcal L}_{\mathfrak D} u$ for each $u \in [H^1 (D)]^m$. Hence, for
$u,v \in [H^1 (D,S)]^m$ we have:
\begin{equation} \label{eq.h}
(u, v)_h  = (P  u, P  v)_{[L^2 (\partial D\setminus S)]^m} + ({\mathfrak D} ^{(3)} u , {\mathfrak
D}^{(3)} v )_{[L^2 (D)]^k}.
\end{equation}
On the other hand, integrating by parts,  we obtain
\[
({\mathfrak D}^{(3)}  P_1  u , {\mathfrak D}^{(3)} G_1 {\mathcal L}_{{\mathfrak D}^{(3)}}
u )_{[L^2 (D)]^k} =0.
\]
That is why, for all $u \in [H^1 (\overline D,S)]^m$,
\[
C_2^2 C_1^{-2}\|u\|^2_{+,\gamma} \geq \|u \|^2_h  \geq
\|P_1  u\|^2_{[L^2 (\partial D\setminus S)]^m} +
\|{\mathfrak D}^{(3)} G_1 {\mathcal L}_0 u \|^2_{[L^2 (D)]^k}
\]
It follows  from(\ref{eq.Ga}) and  (\ref{eq.h}) that any  sequence
$\{u_\mu \} \subset [H^{1} (D,S)]^m$,  converging to  $u \in H^{+,\gamma}_ {{\mathfrak D}^{(3)}}
(D)$ in the space $H^{+,\gamma}_{{\mathfrak D}^{(3)}}(D)$ can be presented as
\[
u_\mu = P_1u_\mu + G _1 {\mathcal L}_{{\mathfrak D}^{(3)}} u_\mu
\]
where the sequence   $\{G _1{\mathcal L}_{{\mathfrak D}^{(3)}}  u_\mu \}$ converges in [$H^1_0
(D)]^m \subset [H^1 (D,S)]^m$ to an element  $w_1$.

Now the already proved part of the theorem yields that $\{P_1 u_\mu \}$ converges to an element
$w_2$ in $[H^{1/2-\varepsilon} (D)]^m$. This proves the continuous embedding  $H^{+,\gamma}
_{{\mathfrak D}^{(3)}}(D) \hookrightarrow [H^{1/2 -\varepsilon} (D)]^m$.

To finish the proof of the theorem one has to almost literally repeat the corresponding
arguments in the proof of \cite[Theorem 1]{PolkShla13}, related to the mixed problem
for the Laplace operator.
\end{proof}

\begin{corollary}
\label{c.SL.emb.half}
Let $\mu$, $\lambda \in C^\infty (X)$ and estimates (\ref{eq.a}),  (\ref{eq.b}) hold true. Then
$H^{+,\gamma} _{{\mathfrak D}^{(3)}} (D)$ is continuously embedded into
$[H^{1/2-\varepsilon,\gamma} (D)]^m$ for any $\varepsilon >0$.
\end{corollary}

The embeddings, described in Theorem  \ref{t.emb.half} and Corollary \ref{c.SL.emb.half},
are rather sharp on the scale of the Sobolev-Slobodetskii spaces (see Example \ref{ex.d3} below).

\begin{remark} \label{r.A.D}
Lemma \ref{l.factor.1}, Theorem  \ref{t.emb.half} and Corollary \ref{c.SL.emb.half} imply that
in the spaces  $H^{+,\gamma} _{{\mathfrak D}^{(1)}} (D)$ and  $H^{+,\gamma} _{{\mathfrak
D}^{(2)}} (D)$ we can use arbitrary first order perturbations $a_1 \nabla_m \otimes I_m$ in
(\ref{eq.A.D}) with $\rho a_1^{(p,q)}\in L^{\infty} (D)$ while for the operators ${\mathfrak
D}^{(3)}$ only the summands of the type $\tilde a_1 (x) {\mathfrak D}^{(3)}$, where  $\tilde a_1$
is a $(m\times k_3)$-matrix with entries $\tilde a_1^{(p,q)}$  satisfying $\rho \tilde
a_1^{(p,q)}\in L^{\infty} (D)$, can be used.
\end{remark}

Now we proceed with the generalized formulation of the Sturm-Liouville Problem.
With this aim, we assume that $H^{+,\gamma}_{\mathfrak D} (D)=H^+$ is continuously
embedded into $H^{0,\gamma} (D)=H^0$ (the corresponding conditions were described above)
and we denote by $H^{-,\gamma}_{\mathfrak D} (D)=H^-$ the completeness of the space  $[H^1
(D,S)]^m$ with respect to the corresponding negative norm  $\| u \|_{-,\gamma,{\mathfrak D}}$.
The pairing, described in Lemma \ref{l.dual}, will be denoted by
$\langle \cdot , \cdot \rangle_\gamma$.

Further, on integrating by parts we see that
\[
   (Au, v)_{[H^{0,\gamma} (D)]^m}
 = ({\mathfrak D}u, {\mathfrak D}v)_{[H^{0,\gamma} (D)]^k}
 + \left( b_1^{-1} (b_0 + \partial_\tau)  u, v \right)_{[H^{0,\gamma} (\partial D \setminus S)]^m}
 + 
\]
\[
 \Big( a_1 \nabla_m \otimes I_m u - 2\gamma\rho^{-1} ({\mathfrak D} \rho)^*  {\mathfrak D}u  +
 a_0 u, v  \Big)_{[H^{0,\gamma} (D)]^m}
\]
for all  $u \in [H^2 (D, S)]^m$ and $v \in [H^1 (D,S)]^m$, satisfying the boundary condition of
(\ref{eq.SL}); here ${\mathfrak D}  \rho $ stands for the functional matrix
$\sum_{j=1}^m {\mathfrak D}  _j \frac{\partial \rho}{\partial x_j}$.
Suppose that (cf. Remark \ref{r.A.D})
\begin{equation}
\label{eq.positive.part1}
\Big|   \left( b_1^{-1} (\delta b_0 + \partial_\tau) u, v \right)_{[L^2 (\partial D
\setminus S)]^m} + \Big(  a_{1}  \nabla_m \otimes I_m  u + \delta a_0\, u, v  \Big)_{[L^2 (D)]^m}
 \Big| \leq   c\, \| u \|_{+,\gamma, {\mathfrak D}} \| v \|_{+,\gamma, {\mathfrak D}}
\end{equation}
for all $u, v \in  [H^1 (D,S\cup S)]^m$, with a positive constant  $c$ being
independent of $u$ and $v$.

Under condition (\ref{eq.positive.part1}), for each fixed
$u \in H^{+,\gamma}_{\mathfrak D} (D)$ the sesquilinear form
\[
Q (u, v) = ({\mathfrak D}u, {\mathfrak D}v)_{[H^{0,\gamma} (D)]^l}
 + \left( b_1^{-1} b_0  u, v \right)_{[H^{0,\gamma} (\partial D \setminus S)]^m} +
\]
\[
 \Big(  a_1 \nabla_m \otimes I_m u- 2\gamma\rho^{-1} ({\mathfrak D}  \rho)^* {\mathfrak D}u  +
 a_0 u, v  \Big)_{[H^{0,\gamma}  (D)]^k}
\]
determines a continuous linear functional $f$ on $H^{+,\gamma} (D)$ via the
equality $f (v) := \overline{Q (u,v)}$ for $v \in H^{+,\gamma} (D)$. By Lemma
\ref{l.dual}, there is a unique element $Lu$ in $H^{-,\gamma} (D)$ such that
\[
   f (v) = \langle v,Lu \rangle _{\gamma}
\]
for all $v \in H^{+,\gamma} (D)$. We have thus defined a linear operator
$L : H^{+,\gamma} (D) \to H^{-,\gamma} (D)$. It follows from
(\ref{eq.positive.part1}) that the operator $L$ is bounded. The bounded linear operator
$L_0 : H^{+,\gamma} (D) \to H^{-,\gamma} (D)$ defined in this way via the sesquilinear
form  $(\cdot,\cdot)_{+,\gamma}$, i.e.,
\begin{equation}
\label{eq.L0}
   (v,u)_{+,\gamma,_{\mathfrak D}} = \langle v, L_0 u \rangle _\gamma
\end{equation}
for all $u, v \in H_{+,\gamma} (D)$, corresponds to the case $a_1 =  \rho^{-1}{\mathfrak D} ^*
\rho $,    $a_0 = a_{0,0}$ and  $b_0 = b_{0,0}$.

Thus, the generalized setting of the problem (\ref{eq.SL}) in the weighted spaces is the
following: given $f \in H^{-,\gamma} (D)$ find $u \in H^{+,\gamma} (D)$ such that
\begin{equation}
\label{eq.SL.w}
\overline{Q (u,v)} = \langle v,f \rangle_\gamma \mbox{ for all } v \in H^{+,\gamma} (D).
\end{equation}

The problem (\ref{eq.SL.w}) can be investigated by the standard methods of functional analysis
\cite[Ch.~3, \S \S~4--6]{LadyUral73}) that are similar to the coercive case.

\begin{lemma}
\label{l.solv.SL1}
Suppose that  $H^{+,\gamma}_{\mathfrak D} (D)$ is continuously embedded into $H^{0,\gamma} (D)$,
$a_1 = 2\gamma\rho^{-1}({\mathfrak D} \rho)^* $,  $\delta a_0 = 0$ and  $\delta b_0 = 0$. Then
for each $f \in  H^{-,\gamma} _{\mathfrak D}(D)$ there is a unique solution $u \in H^{+,\gamma}
_{\mathfrak D}(D) $ to problem (\ref{eq.SL.w}), i.e., the operator $L_0: H^{+,\gamma}_{\mathfrak
D} (D) \to H^{-,\gamma} _{\mathfrak D} (D)$ is continuously invertible. Moreover,
the norms of the operators $L_0$ and $L_0^{-1}$ equal to $1$.
\end{lemma}

The following three lemmas describe bounded and compact perturbations of the operator $L_0$.

\begin{lemma}
\label{l.bound.comp.1} Let $H^{+,\gamma}_{\mathfrak D} (D)$ be continuously embedded into
$H^{s,\gamma} (D)$, $0<s\leq 1$. If $\rho a_1 \in L^\infty (D)$, $\rho ^2 \delta a_0
\in L^\infty (D)$ then the corresponding summands in problem (\ref{eq.SL.w}) induce bounded
operators, acting from $H^{+,\gamma}_{\mathfrak D} (D)$ to $H^{-,\gamma} _{\mathfrak D} (D)$.
Moreover, if there is $\varepsilon >0$ such that $\rho^{2-\varepsilon} \delta a_0 \in L^\infty
(D)$,   $\rho^{1-\varepsilon} \delta a_1 \in L^\infty (D)$ then the corresponding summands in
problem \ref{eq.SL.w}) induce compact operators, acting from from $H^{+,\gamma}_{\mathfrak D}
(D)$ to $H^{-,\gamma} _{\mathfrak D} (D)$.
\end{lemma}

\begin{proof} It follows from Lemma \ref{eq.l.emb.s.gamma}.
\end{proof}

In the coercive case (corresponding to the operators ${\mathfrak D}^{(1)}$, ${\mathfrak
D}^{(2)}$) we can enlarge the class of the perturbations. With this purpose we fix a basis $\{
t_j\}_{j=1}^{m-1}$ among the tangential vectors (with bounded integrable components). For instance, this may be formed by the vectors
\begin{equation} \label{eq.basis}
\vec{e}_j \nu_i - \vec{e}_i \nu_j, \quad i>j.
\end{equation}
Then
$\partial _\tau = \sum_{j=1}^{m-1} d_j (x)\partial_{ t_j}$ with $(m\times m)$-matrices
$d_j (x)$.

\begin{lemma}
\label{l.bound.comp.2} Let $j=1$ or $j=2$. Let (\ref{eq.a}) hold or $\rho\equiv 1$. If
$\rho b_1^{-1}\delta b_0 \in L^\infty (\partial D\setminus S)$  then the corresponding
summand in problem (\ref{eq.SL.w}) induces a bounded operator, acting from $H^{+,\gamma}
_{{\mathfrak D}^{(j)}} (D)$ to $H^{-,\gamma} _{{\mathfrak D}^{(j)}}  (D)$.  if there is
$\varepsilon >0$ such that $\rho^{1-\varepsilon} b_1^{-1}\delta b_0 \in L^\infty (\partial
 D\setminus S)$ then the corresponding summand in problem (\ref{eq.SL.w}) induces a compact
 operator, acting from $H^{+,\gamma}_{{\mathfrak D}^{(j)}} (D) \to H^{-,\gamma} _{{\mathfrak
 D}^{(j)}}  (D)$. Moreover, if $b_1^{-1}d_j \in C^{0,\alpha} (\partial D \setminus S) $, $1\leq j
 \leq m-1$,  $1/2<\alpha \leq 1$ then the matrix $\partial _\tau$ of tangential derivatives
 induces a  bounded operator, acting from $H^{+,\gamma}_{{\mathfrak D}^{(j)}} (D)$ to
 $H^{-,\gamma}  _{{\mathfrak  D}^{(j)}}  (D)$ with the norm estimated via
 $\|b_1^{-1}d_j\|_{C^{0,\alpha}  (\partial D\setminus S)}$,  $1\leq j \leq m-1$.
\end{lemma}

\begin{proof} The continuity and the compactness of the operators induced by
the summand  $b_1^{-1} \delta b_0$ follows from Lemma \ref{l.factor.1}, the Embedding Theorem for
Sobolev spaces and the continuity of the trace operator $tr: H^{1,\gamma} (D) \to H^{1/2,\gamma}
(\partial D)$ (see Lemma \ref{eq.l.emb.s.gamma}).

In order to finish the proof of the continuity of the tangential operator
one has to almost literally repeat the corresponding arguments in the proof
of \cite[Lemma 6.6]{TarShla13a}, related to the similar mixed problem
for the scalar differential operators.
\end{proof}

\begin{lemma}
\label{l.bound.comp.3} Let inequality  (\ref{eq.b}) be fulfilled and $b_1^{-1}\delta b_0
\in L^\infty (\partial D\setminus S)$. If (\ref{eq.a}) is true or $\rho \equiv 1$ then
the corresponding summand in problem (\ref{eq.SL.w}) induces a bounded  operator, acting from
$H^{+,\gamma}_{{\mathfrak D}^{(3)}} (D)$ to $H^{-,\gamma} _{{\mathfrak D}^{(3)}}  (D)$.
\end{lemma}

As examples \cite[Examples 1,2]{PolkShla13}  show, the boundary terms $\delta b_0 $ and
$\partial _\tau$ do not induce compact and bounded perturbations respectively for  $L_0$
if $m=2$ and ${\mathfrak D}={\mathfrak D}^{(3)}$.

Now we split
\[
\delta b_0 = \delta b_0^{(s)} + \delta b_0^{(c)}, \quad \delta a_0 = \delta a_0^{(s)} + \delta
a_0^{(c)}, \quad  a_1 \nabla_m \otimes I_m = 2\gamma \rho^{-1} ({\mathfrak D}\rho)^* {\mathfrak
D} + (\delta a_1^{(s)} + \delta a_1^{(c)}) \nabla_m \otimes I_m,
\]
in such a way that the terms $\delta b_0^{(c)}$, $\delta a_0^{(c)}$ and $\delta a_1^{(c)}$ induce
the compact perturbations of the operator  $L_0$ and the summands $\delta b_0^{(s)}$, $\delta
a_0^{(s)}$ and  $\delta a_1^{(s)}$ induce the small ones. This gives the possibility to use the
perturbation methods.

The proof of the following two statements is standard (see, for example,
\cite{LiMa72}, \cite{Mikh76}, \cite{TarShla13a}).

\begin{theorem}
\label{t.solv.SL1}
Let $j=1$ or $j=2$. Let  $d_j \in
C^{0,\lambda} (\partial D \setminus S) $, $1\leq j \leq m-1$. Besides, let
(\ref{eq.a}) hold or $\rho\equiv 1$. If there exists  $\varepsilon >0$ such that
$\rho^{2-\varepsilon} \delta a_0^{(c)} \in L^\infty
(D)$,   $\rho^{1-\varepsilon} \delta a_1^{(c)} \in L^\infty (D)$,
$\rho^{1-\varepsilon} \delta b_0^{(c)} \in L^\infty (\partial D\setminus S)$, and
\begin{equation}
\label{eq.positive.part11}
|   ( b_1^{-1} (\delta b^{(s)}_0 + \partial_\tau) u, v )_{[L^2 (\partial D \setminus
 S)]^m} + (  \delta a_1^{(s)} \nabla_m \otimes I_m u + \delta a_0^{(s)}\, u, v  )_{[L^2 (D)]^m}
   | \leq   M\, \| u \|_{+,\gamma, {\mathfrak D}^{(j)}} \| v \|_{+,\gamma,  {\mathfrak D}^{(j)}}
\end{equation}
for all $u, v \in  [H^1 (D,S\cup S)]^m$ with a constant $0<M<1$ being independent
on $u$ and $v$ then problem (\ref{eq.SL.w}) is a Fredholm one.
\end{theorem}

For ${\mathfrak D}^{(3)}$ we split in a different way:
$\tilde a_1 = 2\gamma \rho^{-1} ({\mathfrak D}^{(3)}\rho)^* + \tilde \delta a_1^{(s)} +
\delta \tilde a_1^{(c)}$ (cf. Remark \ref{r.A.D}).

\begin{theorem}
\label{t.solv.SL2}
Let  ${\mathfrak D}= {\mathfrak D}^{(3)}$,  estimate (\ref{eq.b}) hold, $\lambda, \mu $
are infinitely smooth in a neighborhood of $\overline D$, $\tau=0$ and $\delta b_0^{(c)}=0$.
Besides, let (\ref{eq.a}) hold or $\rho \equiv 1$. If there exists  $\varepsilon >0$ such that
$\rho^{2-\varepsilon} \delta a_0^{(c)} \in L^\infty (D)$,   $\rho^{1-\varepsilon}
\delta \tilde a_1^{(c)} \in L^\infty (D)$,  and
\begin{equation}
\label{eq.positive.part12}
   |( b_1^{-1} \delta b_0^{(s)} u, v )_{[L^2 (\partial D \setminus S)]^m}
 + (  \delta \tilde a_1^{(s)} {\mathfrak D}^{(3)} u + \delta a_0^{(s)}\, u, v  )_{[L^2 (D)]^m}  |
 \leq \tilde M  \| u \|_{+,\gamma, {\mathfrak D}^{(3)}} \| v \|_{+,\gamma, {\mathfrak D}^{(3)}}
\end{equation}
for all $u, v \in  [H^1 (D,S\cup S)]^m$ with a constant $0<\tilde M<1$, being independent
on $u$ and $v$ then problem (\ref{eq.SL.w}) is a Fredholm one.
\end{theorem}

\section{The spectral properties of the mixed problems}

In this section we use Theorems  \ref{t.solv.SL1}, \ref{t.solv.SL2} and the standard
tools of Functional Analysis for the description of the completeness of the root elements of the
mixed problem (\ref{eq.SL.w}) in the spaces $H^{+,\gamma}_{\mathfrak D} (D)$, $[H^{0,\gamma}
(D)]^m$ and $H^{-,\gamma}_{\mathfrak D}(D)$. We  study  both the coercive and the
non-coercive cases.

With this aim we consider the sesquilinear form
\[
(u,v)_{-,\gamma, {\mathfrak D}}  := \langle L_0^{-1} u, v \rangle_\gamma \mbox{ for } u, v \in
H^{-,\gamma}_{\mathfrak D}(D),
\]
on the space $H^{-,\gamma}_{\mathfrak D} (D)$. It is well known that
$ \sqrt{(u, u)_{-,\gamma, {\mathfrak D}}} = \| u \|_{-,\gamma, {\mathfrak D}}$ for all
$u \in H^{-,\gamma}_{\mathfrak D}(D) (D)$. From now on we endow the space $H^{-,\gamma}
_{\mathfrak D}(D)$ with the scalar product $(\cdot,\cdot)_{-,\gamma, {\mathfrak D}}$.

We recall  that a compact self-adjoint operator $C$ is said to be of finite order if
there is $0 < p < \infty$, such that the series $ \sum_\nu |\lambda_\nu|^p $ converges
where  $\{ \lambda_\nu \}$ is the system of eigenvalues of the operator $C$
(here the summation is done counting the multiplicities of the eigenvalues, see, for instance,
\cite{GokhKrei69} and elsewhere).

\begin{theorem}
\label{t.L0.selfadj}
If  $H^{+,\gamma}_{\mathfrak D} (D)$ is continuously embedded into $H^{0,\gamma} (D)$ then
 the inverse $L^{-1}_0$ of the operator given by (\ref{eq.L0}) induces  positive
self-adjoint operators
\[
\iota' \iota\, L^{-1}_{0}: H^{-,\gamma}_{\mathfrak D}(D)  \to H^{-,\gamma}_{\mathfrak D}(D) ,
\quad    \iota\, L^{-1}_{0}\, \iota' :   [H^{0,\gamma} (D)]^m \to [H^{0,\gamma} (D)]^m ,
\]
\[
  L^{-1}_{0}\, \iota' \iota  : H^{+,\gamma}_{\mathfrak D}(D)\to  H^{+,\gamma}_{\mathfrak D}(D),
\]
which have the same systems of eigenvalues and eigenvectors; besides, the eigenvalues are
positive. Moreover, if $H^{+,\gamma}_{\mathfrak D}(D)$ is continuously embedded into
$H^{s,\gamma}(D)$ with $0<s\leq 1$ then they are compact operators of finite orders
 and there are orthonormal basis in the spaces
   $H^{+,\gamma}_{\mathfrak D}(D)$,
   $[H^{0,\gamma} (D)]^m$ and
   $H^{-,\gamma}_{\mathfrak D}(D)$.
\end{theorem}

\begin{proof} The first part of the theorem is well-known (see, for instance,
\cite{LiMa72}, \cite{Mikh76}, \cite{TarShla13a}). Besides,
\begin{equation} \label{eq.orth.+}
(\iota' \iota\, L^{-1}_{0} u, v)_{-,\gamma,{\mathfrak D}}  =
 (\iota L^{-1}_{0} u , \iota L^{-1}_{0} v)_{[H^{0,\gamma} (D)]^m},
 (L^{-1}_{0}\, \iota' \iota u, v)_{+,\gamma,{\mathfrak D}}  =    (\iota u, \iota
 v)_{[H^{0,\gamma} (D)]^m}.
\end{equation}
\[
   (\iota\, L^{-1}_{0}\, \iota' u, v)_{[H^{0,\gamma} (D)]^m}
  =    (L^{-1}_{0} (\iota' u), L^{-1}_{0} (\iota' v))_{+,\gamma,{\mathfrak D}}.
\]
Moreover, Lemma \ref{eq.l.emb.s.gamma} implies that under the hypothesis of the lemma,
the operator $\iota$ is compact. Therefore the statement on the basis
follows from the Hilbert-Schmidt theorems and the identities (\ref{eq.orth.+}).
That is why it is left to prove the statement on operator's orders only.

For the usual Sobolev spaces the statement follows from results of \cite{Agmo62} (see also
\cite[Theorem 3.2]{TarShla13a}), because in this situation the operator
$\iota' \iota\, L^{-1}_{0}$ maps, in fact, $[H^{-s} (D)]^m \subset H^{-,\gamma}_{{\mathfrak
D}^{(j)}} (D)$ to $H^{+,\gamma}_{{\mathfrak D}^{(j)}} (D) \subset [H^{s} (D)]^m$, $j=1,2$.

Since the embedding $H^{s,0} (D) \to H^{s} (D)$ is obviously bounded, then
for the weighted Sobolev spaces the correspondence $u \mapsto \rho^{-\gamma} u$
induces a continuous map $S^+: H^{s,\gamma} (D) \to H^{s} (D)$, and the correspondence
$v \mapsto \rho^{\gamma} v$ induces a continuous map $S^-: H^{-s} (D) \to
H^{-s,\gamma} (D) $. Hence, if the embedding $i_s: H^{+,\gamma} (D) \to H^{s,\gamma} (D)$
is continuous then the results of \cite{Agmo62} imply that the order of the operator $\iota'_s
\iota_s \, S^+ i_s L^{-1}_{0} i'_s S^-:  H^{-s} (D) \to  H^{-s} (D)$ equals to $2s$ and it
has the same eigenvalues as the operator $\iota' \iota\, L^{-1}_{0} $ (here
$\iota_s: H^s (D) \to L^2 (D)$ is the natural embedding. 
\end{proof}

It is not difficult to show that the operator $L: H^{+,\gamma} _{{\mathfrak D}} (D) \to
H^{-,\gamma} _{{\mathfrak D}} (D)$ induces a closed densely defined linear operator $T :
H^{-,\gamma} _{{\mathfrak D}} (D) \to H^{-,\gamma} _{{\mathfrak D}} (D)$ with the domain
$H^{+,\gamma} _{{\mathfrak D}} (D)$. The the operator $L_0$ corresponds to a symmetric closed
operator $T_0 : H^{-,\gamma} _{{\mathfrak D}} (D) \to H^{-,\gamma} _{{\mathfrak D}} (D)$ having
the same eigenvectors as the operator $\iota' \iota\, L^{-1}_{0}
: H^{-,\gamma} _{{\mathfrak D}} (D) \to H^{-,\gamma} _{{\mathfrak D}} (D)$.
As it is known, non-selfadjoint operators in infinite-dimensional spaces  may have not enough
eigenvectors to form a basis. Hence the notion of the root vectors is very important.

Recall that a non-zero vector  $u$ from the domain  $ D (T)$ of a linear operator
$T$ on a linear space $H$ is called a root vector (or, the generalized
eigenvector) for $T$, if there are numbers $N \in \mathbb N$ and
$\lambda \in \mathbb C$ satisfying $(T-\lambda I)^N u =0$, where $I: H \to H $ is the identity
operator in $H$.

The conditions providing the completeness of the root vectors are well known in the frames of the
functional analysis (see \cite{Agra11a}, \cite{Agra11c}, \cite{GokhKrei69},
\cite{Keld51} and others).

\begin{corollary}
\label{c.root.func.1}
Under the hypotheses of Theorem \ref{t.solv.SL1}, if $M< \sin{\pi /m}$ then the system of the
root vectors of the closed operator $T$ is complete in the spaces $H^{-,\gamma}_{{\mathfrak
D}^{(j)}} (D)$, $[H^{0,\gamma} (D)]^m$ and $H^{+,\gamma}_{{\mathfrak D}^{(j)}} (D)$, $j=1,2$.
Moreover, for any $\delta > 0$ all the eigenvalues of $T$ (except a finite number of them) belong
to the angle $|\arg \lambda| <  \delta  + \arcsin M$ in $\mathbb C$.
\end{corollary}

\begin{proof} Follows from Theorems \ref{t.solv.SL1}, \ref{t.L0.selfadj} and the Spectral Theory
of non-selfadjoint operators  (see, for instance, \cite{Agra11a}, \cite{Agra11c}, \cite{Brow59b},
\cite[Theorem 6.8]{TarShla13a}).
\end{proof}

\begin{corollary}
\label{c.root.func.2}
Under the hypotheses of Theorem \ref{t.solv.SL2}, if $\tilde M< \sin{\pi /2m}$,
$H^{-,\gamma}_{{\mathfrak D}^{(3)}} (D)$, $[H^{0,\gamma} (D)]^m$ and $H^{+,\gamma} _{{\mathfrak
D}^{(3)}} (D)$. Moreover, for any $\delta > 0$ all the eigenvalues of $T$ (except a finite number
of them) belong to the angle $|\arg \lambda| <  \delta  + \arcsin \tilde M$ in $\mathbb C$.
\end{corollary}

\begin{proof}
Follows from Theorems \ref{t.solv.SL2}, \ref{t.L0.selfadj} and the Spectral Theory of
non-selfadjoint operators  (see, for instance, \cite[Theorem 4.5]{TarShla13a}).
\end{proof}

\begin{example}  \label{ex.d1} {\rm
Let $\rho \equiv 1$. The mixed problem (\ref{eq.SL.w}) for $A= ({\mathfrak D}^{(1)})^* {\mathfrak
D}^{(1)} $ and $B=\chi_S +  \chi_{\partial D\setminus S} \sigma $ is classical in the
Elasticity Theory (see \cite[\S 12]{Fi72});  here $\chi_{M}$ is the characteristic function
of the set $M$. As the corresponding sesquilinear form $(\cdot , \cdot )_{+,\gamma,{\mathfrak
D}^{(1)}}$ is coercive for $b_{0,0}= 0$, $\mu \geq \kappa>0$,  $\lambda \geq 0$  (see Lemma
\ref{l.factor.1}),  we may also consider the boundary operators $\chi_S + \chi_{\partial
D\setminus S} (\sigma  + T (x)\partial _{\tau_0} + \delta b_0) $  with a matrix $T$ having small
entries of the class $C^{0,\alpha} (\partial D\setminus S)$, $1/2 < \alpha \leq 1$ and with
the perturbation $\delta b_0$ described in Theorem \ref{t.solv.SL1}. The low order
acceptable perturbations  are also indicated in Theorem \ref{t.solv.SL1}.
The completeness conditions are described in Corollary \ref{c.root.func.1}.}
\end{example}

\begin{example}  \label{ex.d3} {\rm
Let  $D$ be the unit circle in ${\mathbb R}^2 (\cong {\mathbb C})$, ans $S$ be that part
of its boundary where $\arg{(z)}  \in [0,2\pi] \setminus [-\pi/2, \pi/2]$. For $\rho \equiv
1$, consider mixed problem  (\ref{eq.SL.w}) with $A=  ({\mathfrak D}^{(3)})^* {\mathfrak
D}^{(3)}$ and $B=\chi_S + \chi_{\partial D\setminus S} (\nu_{{\mathfrak D}^{(3)}} + b_{0,0})$,
see (\ref{eq.Lame}), (\ref{eq.stress.tensor}). Let $\mu \geq \kappa>0$,  $\mu+\lambda \geq 0$
be constants  and $b_{0,0}>0$ be a matrix with constant entries. Let $\varphi  \in C^\infty
(\overline D) $ equal to zero identically in a neighborhood of $S$ and equal to one on the part
where $\arg{(z)}  \in [-\pi/4, \pi/4]$. It is clear that the function  $u_\varepsilon  = \phi
(z)(\Re{v_\varepsilon}, \Im{v_\varepsilon})$, where
\[
v_\varepsilon (z) = \sum_{\nu=0}^\infty \frac{{\overline z}^{4\nu}}{(4\nu
+1)^{(1+\varepsilon)/2}},
\]
belongs to $H^{+,\gamma} (D)$ but for any  $s\in (1/2,1]$ there is $\varepsilon >0$ such that
$ u_\varepsilon \not \in H^s (D)$ (cf. \cite[Examples 1,2]{PolkShla13}). Thus,
$H^{+,\gamma}_{{\mathfrak D}^{(3)}} (D)$ is continuously embedded into  $[H^{1/2} (D)]^2$
(see Theorem  \ref{t.emb.half}), but it is not embedded into $[H^s (D)]^2$ for any
$s\in (1/2,1]$ (cf.  \cite[Examples 1, 2]{PolkShla13}). Moreover, as $Y \subset \partial S$, then
this example can be easily adopted to the weighted spaces.

Clearly, $\nu_{{\mathfrak D}^{(3)}}$ is responsible not for the stress/viscosity on the boundary
but for a more large class of interactions with $\partial D$. For instance, interpreting
the Lam\'e system as a linearization of the stationary version of the Navier-Stokes' type
equations for the compressible fluids, we see that the boundary operator $(\nu_{{\mathfrak
D}^{(3)}} + b_{0,0})$ reflects rather the vorticity and the source density on conormal directions
to  $\partial D \setminus S$. This means that the boundary operator $\nu_{{\mathfrak D}^{(3)}}$
is more fit to study problems, related  to models with the turbulent flows, than the operators
$\nu_{{\mathfrak D}^{(1)}}$ and $\nu_{{\mathfrak D}^{(2)}}$. Then it is natural that the class of
the possible solutions to (\ref{eq.SL.w}) extends up to $H^{+,\gamma} _{{\mathfrak D}^{(3)}} (D)$
due to the loss of the regularity of solutions near $\partial D\setminus S$.}
\end{example}

\begin{example}  \label{ex.d2} {\rm
Let $\rho \equiv 1$. Consider mixed problem (\ref{eq.SL.w}) for $A= ({\mathfrak D}^{(2)})^*
{\mathfrak D}^{(2)}$ and $B=\chi_S +  \chi_{\partial D\setminus S} (\nu_{{\mathfrak D}^{(2)}}  +
h \mu (x)\partial _{\tau_0}) $ with a small parameter $h$,  see  (\ref{eq.Lame}),
(\ref{eq.stress.tensor}),  in the case where $\mu \geq \kappa>0$,  $\mu+\lambda \geq \kappa>0$.
In particular, if we choose the vectors (\ref{eq.basis}) as a basis among the tangential vectors
to $\partial D$ then $d_j (x) = h  \chi_{\partial D\setminus S}   \mu (x) I_m$, $1\leq j
\leq m-1$.

Assume that $\mu , \lambda\in C^{0,1} (\overline D)$. Then  $\mu \in L^{\infty} (D)$,
$\nabla_m \mu \in L^{\infty} (D)$, $\mu \in C^{0,\alpha} (\partial D\setminus S)$ for all
$1/2 < \alpha \leq 1$. According to Lemma \ref{l.factor.1}, the norms of the spaces
$H^{+,\gamma}_{{\mathfrak  D}^{(2)}} (D)$ and $[ H^{1} (D)]^m$ are equivalent.  Lemma
\ref{l.bound.comp.1} implies that the first order terms induce the compact operators,
acting from  $[H^{1} (D)]^m$ to $[H^{-1} (D)]^m$. If the value $|h|$ is sufficiently small
then problem (\ref{eq.SL.w}) is a Fredholm one and its root vectors are dense in $[ H^{1}
(D)]^m$, $[ H^{-1} (D)]^m$, $[ L^{2} (D)]^m$. If the value $|h|$,
is sufficiently small then problem (\ref{eq.SL.w}) is uniquely solvable.
The other acceptable perturbations are described in Theorem \ref{t.solv.SL1}.

If the coefficients $\lambda, \mu$ are constants then Gau\ss-Ostrogradskii formula implies
\[
|( \partial_{\tau_0} u, v)_{L^{2} (\partial D\setminus S)|} = \left|  \sum_{j=1}^m(\nabla _m
u_j, \partial_j v)_{L^2 (D)} -  (\mbox{div}_m \ u, \mbox{div}_m \ v)_{L^2 (D)} \right|
\mbox{ for all } u,v \in H^1 (D,S),
\]
i.e. $\mu \|\partial _{\tau_0}\|\leq 1$. But it follows from (\ref{eq.stress.var}) that
$ \nu_{{\mathfrak D}^{(3)}} + b_{0,0} = \nu_{{\mathfrak D}^{(2)}} + b_{0,0}- \mu
\partial_{\tau_0}$. Thus, if $\mu \|\partial_{\tau_{0}}\|<1$ then, for  matrices $b_{0,0}$ with
rather small entries, the mixed problem with the boundary operator $\nu_{{\mathfrak D}^{(3)}}
+ b_{0,0}$ can be interpreted as a small perturbation of the mixed problem with the boundary
operator $\nu_{{\mathfrak D}^{(2)}}$. However this contradicts with Example \ref{ex.d2}, because
the space $H^{+,\gamma}_{{\mathfrak D}^{(3)}} (D)$  is not embedded into
$H^{+,\gamma}_{{\mathfrak D}^{(2)}} (D)=[H^{1,\gamma} (D)]^m$. Hence for constant Lam\'e
coefficients the perturbation method is valid with $|h|<1$. In particular, formula
(\ref{eq.stress.var}) means that the mixed problem with
the boundary operator  $\chi_S + \chi_{\partial D\setminus S} \nu_{{\mathfrak D}^{(2)}}  $
can not be investigated as the perturbation  of the mixed problem with the boundary
operator $\chi_S + \chi_{\partial D\setminus S} \sigma $ in the space $[H^{1,\gamma} (D)]^m$.
}
\end{example}

In conclusion, we give examples of proper weight-functions.

\begin{example} Consider the cylinder $D=\{  (x_1, \dots, x_{m-1}) \in \Omega , 0<x_m<1\}$
with the base  $S=\{(x_1, \dots, x_{m-1}) \in \Omega, \, x_m=0\}$ and the set
$Y=\partial S$, where $ \Omega$ is domain with smooth boundary in ${\mathbb R}^{m-1}$.
Let $\phi (x_1, \dots, x_{m-1})$ be the defining function for the domain $ \Omega$, i.e.
it is a real-valued function with $\nabla \phi =1$ on $\partial \Omega$ such that
$\Omega = \{ (x_1, \dots, x_{m-1}) \in {\mathbb R}^{m-1}: \, \phi (x_1, \dots, x_{m-1})<0 \}$.
Then $\rho (x) = \sqrt{\phi^2 (x_1, \dots, x_{m-1}) + x_m^2 }$.
\end{example}

\begin{example} Consider the cube $D=\{ -1 <x_j<1, 1\leq j \leq m-1, 0<x_m<1\}$ with a
distinguished side $S=\{-1 <x_1<1, -1 <x_2<1, -1 <x_j<1, 1\leq j \leq m-1, x_m=0\}$
and the set $Y=\partial S$. In this situation we may set
\[
\rho (x) = \Big( \Pi _{j=1}^{m-1} ((x_j-1)^2  + x_m^2) \, ((x_j+1)^2 + x_m^2) \Big)^{1/2}.
\]
\end{example}

\medskip

\emph{The work was supported by RFBR grants 14-01-00544 and 14-01-00081. 
}


\def\bstname{name}


\begin{thebibliography}{[9]}

\bibitem{Agmo62}
{S.~Agmon},
  On the eigenfunctions and on the eigenvalues of general elliptic boundary value problems,
 Comm. Pure Appl. Math. \textbf{15}, pp.\,119--147 (1962).

\bibitem{ADN59}
{S.~Agmon, A.~Douglis and L.~Nirenberg},
  Estimates near the boundary for solutions of elliptic partial differential equations
  satisfying general boundary conditions, Part 1.  Comm. Pure Appl. Math. \textbf{12}, pp.\,623--727 (1959).
\bibitem{Agra90}
{M.~S. Agranovich},
  Elliptic operators on closed manifold, in: Current Problems of Mathematics, Fundamental Directions, Vol. \textbf{63},  (VINITI, Moscow, 1990), pp.\,5֭129.

\bibitem{Agra11a}
{M.~S.~Agranovich},
   Spectral Problems in Lipschitz Domains, in:
  Modern Mathematics, Fundamental Directions, Vol. \textbf{39} (VINITI, Moscow, 2011), 
   pp.\, 11--35.

\bibitem{Agra11c}
{M.~S. Agranovich},
  Strongly elliptic second order systems with boundary conditions on
          a non-closed Lipschitz surface,
  Funct. Anal. Appl. \textbf{45:1}, pp.\,1֭15 (2011).

 \bibitem{BoKo}
{M.~Borsuk and V.A.~Kondrat'ev},
Elliptic Boundary Value Problems of Second Order in
Piecewise Smooth Domains (Elsevier, Amsterdam-London, 2006).

\bibitem{Brow59b}
{F.~E.~Browder},
  On the spectral theory of strongly elliptic differential operators,
  Proc. Nat. Acad. Sci. USA \textbf{45}, pp.\,1423--1431 (1959).

\bibitem{Ca59} S.~Campanato, Sui problemi al contorno per sistemi
di equazioni differenziale lineari del tipo dell'elasticit\'a, Ann. della Scuola
Norm. Superiore, Cl. di Sci, Ser. III, \textbf{13:2}, pp.\,223--258 (1959).

\bibitem{Ca60}  S.~Campanato, Propriet\'a
di taluni spazi di distribuzioni e loro applicazione, Ann. della Scuola
Norm. Superiore, Cl. di Sci, Ser. III, \textbf{14:4}, pp.\,363--376 (1960).

\bibitem{EgorKondSchu01}
{Yu.~Egorov, V.~Kondratiev, B.-W.~Schulze},
  Completeness of eigenfunctions of an elliptic operator on a manifold
  with conical points, Russ. J. Math. Phys. \textbf{8:3}, pp.\,267--274 (2001).

\bibitem{Eski73}
{G.~I. Eskin},
  Boundary Value Problems for Elliptic Pseudodifferential Operators
  (Nauka, Moscow, 1973).

\bibitem{Fi72}
G.~Fichera,
Existence Theorems in Elasticity, in: Festk\"orpermechanik/Mechanics of
Solids, edited by S.~Fl\"ugge, C.A.~Truesdell, 
Handbuch der Physik (Berlin-׈eidelberg֭New York, Springer--Verlag, 1972), pp.\,347֭389.

\bibitem{GokhKrei69}
 {I.~Ts.~Gokhberg} and {M.~G.~Krein},
  Introduction to the Theory of Linear Nonselfadjoint Operators in Hilbert
 Spaces,  (AMS, Providence, R.I., 1969).

\bibitem{Keld51}
{M.~V.~Keldysh},
  On the characteristic values and characteristic functions of certain
          classes of non-selfadjoint equations,
 Dokl. AN SSSR \textbf{77}, pp.\,11--14 (1951).

\bibitem{KN65}
{J.J.~Kohn and L.~Nirenberg},
  Non-coercive boundary value problems.
  Comm. Pure Appl. Math. \textbf{18}, pp.\,443--492 (1965).

\bibitem{Kond99}
{V.~A.~Kondrat'ev},
  Completeness of the systems of root functions of elliptic operators in Banach
          spaces,
 Russ. J. Math. Phys. \textbf{6:10}, pp.\,194--201 (1999).

\bibitem{Kond66}
{V.~A. Kondrat'ev},
  Boundary problems for parabolic equations in closed domains,
  Trans. Moscow Math. Soc. \textbf{15}, pp.\,400--451 (1966).

\bibitem{LadyUral73}
{O.~A.~Ladyzhenskaya}, and {N.~N.~Uraltseva},
 Linear and Quasilinear Equations of Elliptic Type
 (Nauka, Moscow, 1973).

\bibitem{LaLi59} L.D.~Landau and E.M.~Lifshitz,
Fluid Mechanics, Volume 6 of A Course of Theoretical Physics
(Pergamon Press, London--New York--Paris, 1959).

\bibitem{LiMa72}
{J.L.~Lions and E.~Magenes}, 
Non-Homogeneous Boundary Value Problems and Applications, Vol. 1
(Springer--Verlag, Berlin--Heidelberg--New York, 1972).

 \bibitem{Mikh76}
{V.~P.~Mikhailov},
  Partial Differential Equations
  (Nauka, Moscow, 1976).

\bibitem{Sche60}
{M.~Schechter},
  Negative norms and boundary problems,
  Ann. Math. \textbf{72:3}, pp.\,581--593 (1960).

\bibitem{PolkShla13}
{A.~Shlapunov} and {A.~Polkovnikov},
  On the spectral properties of a non-coercive mixed problem associated with
          $\overline \partial$-operator,  J. Siberian Fed. Uni. \textbf{6:2}, (2013).

\bibitem{TarShla13a} A.~Shlapunov and N.~Tarkhanov,
On completeness of root functions
of Sturm-Liouville problems
with discontinuous boundary operators,
Journal of Differential Equations \textbf{255}, pp.\,3305-3337 (2013).


\bibitem{ShTaLMS}   A.~Shlapunov and N.~Tarkhanov, Bases with double orthogonality in
the Cauchy problem for systems with injective symbols, {Proc. London Math. Soc.} {\bf 71:3}, pp.\,1--52 (1995).

\bibitem{SchShTa03} A.~Shlapunov, N.~Tarkhanov, and B. W. Schulze,
Green integrals on manifolds with cracks,
Annals of Global Analysis and Geometry \textbf{24}, pp.\,131--160 (2003).

\bibitem{Tark06}
{N.~Tarkhanov},
 On the root functions of general elliptic boundary value problems,
 Compl. Anal. Oper. Theory \textbf{1}, pp.\,115--141 (2006).


\end{thebibliography}
\end{document}